\documentclass[a4paper,11pt]{article}
\usepackage[mathscr]{eucal}
\usepackage{amssymb}
\usepackage{latexsym}
\usepackage{amsthm}
\usepackage{amsmath}
\usepackage[dvips]{graphicx}
\usepackage{psfrag}

\DeclareMathOperator{\adj}{{\rm adj}}

\newtheorem{theorem}{Theorem}[section]

\newtheorem{lemma}[theorem]{Lemma}
\newtheorem{corollary}[theorem]{Corollary}
\theoremstyle{definition}

\newtheorem{remark}[theorem]{Remark}

\makeatletter
\@addtoreset{equation}{section}

\makeatother

\title{A characterization of skew Hadamard matrices and doubly regular tournaments}
\author{
Hiroshi Nozaki\footnote{Research supported by JSPS Research Fellowship. 
}
and 
Sho Suda
}
\begin{document}
\maketitle

\renewcommand{\thefootnote}{\fnsymbol{footnote}}
\footnote[0]{2010 Mathematics Subject Classification: 05B20}
\begin{abstract}
We give a new characterization of skew Hadamard matrices of size $n$ in terms of the data of the spectra of tournaments of size $n-2$.
\end{abstract}

\section{Introduction}
A square matrix $H$ of size $n$ with entries $\pm1$ is a {\it Hadamard matrix} if $H H^T=n I$.
Hadamard matrices of size $n$ exist only if $n$ is two or a multiple of four.
A Hadamard matrix $H$ is said to be {\it skew} if $H+H^T=2I$, where $I$ denotes the identity matrix.
It was shown that the existence of the following are equivalent:
\begin{enumerate}
\item\label{1} Skew Hadamard matrices of size $n$.
\item\label{2} Doubly regular tournaments of size $n-1$ \cite{RB1972}.
\item\label{3} Irreducible tournaments of size $n$ having $4$ eigenvalues, one of which is zero with the algebraic multiplicity $1$ \cite{KB1994}.
\item\label{4} Tournaments of size $n-1$ with spectrum $\{k^1,\big(\tfrac{-1+\sqrt{-k}}{2}\big)^k,\big(\tfrac{-1-\sqrt{-k}}{2}\big)^k\}$, where $n=2k+2$ \cite{Z}
\end{enumerate}
For a skew Hadamard matrix $H$, we normalize $H$ so that the first row of $H$ consists of the all-ones vector.
We can construct a $(0,1)$-matrix $A$ by $A=\frac{1}{2}(J-H)$, where $J$ denotes the all-ones matrix.
Then $A$ is the adjacency matrix of a tournament satisfying the condition \eqref{3}.
For such a matrix $A$, we consider the principal submatrix $A_1$ of size $n-1$ by deleting the first row and column.
Then $A_1$ satisfies the condition \eqref{4} and \eqref{2},--and vice versa.
Thus skew Hadamard matrices of size $n$ are characterized by some tournaments of size $n-1$.

As described in the above, the characterization for some property of an oriented graph in terms of the spectrum of its adjacency matrix is important and useful.
For example, a tournament is regular if and only if its adjacency matrix has the all-ones eigenvector.

In the algebraic graph theory, the adjacency matrix plays an important role \cite{CRS1997,GR2001}.
The adjacency matrix of an undirected graph is diagonalizable. 
However that of an oriented graph is not necessarily diagonalized, 
and hence the dealing with the adjacency matrix for an oriented graph is more difficult than the case of an undirected graph.
In the area of two-graphs the Seidel matrix, namely so-called a $(0,\pm1)$-adjacency matrix, is used \cite[Section 11]{GR2001}.
The Seidel matrix for an oriented graph is defined naturally, and
since it is always Hermitian, it is easy to use the Seidel matrix in the case of an oriented graph.   

In the present paper, we give another characterization of a skew Hadamard matrix of size $n$ in terms of the spectrum of the Seidel matrix of a tournament of size $n-2$.
Our main theorem is as follows: 
\begin{theorem}\label{thm:1}
Let $n=4k+3$, where $k$ is a non-negative integer.
Then there exists a doubly regular tournament of size $n$ if and only if there exists a tournament of size $n-1$ with the adjacency matrix $A_1$ such that $S_1=\sqrt{-1}(A_1-A_1^T)$ satisfies the following spectrum condition: 
\begin{align}\label{eq:1}
 (\tilde{\theta}_i)_{i=1}^4=(\sqrt{n},1,-1,-\sqrt{n}) \text{ with } \tilde{\beta}_1=\tilde{\beta_4}=0, \tilde{\beta}_2=\tilde{\beta}_3=\frac{1}{\sqrt{2}}.
\end{align}
\end{theorem}
Here $\tilde{\theta}_i$ ($1\leq i\leq 4$) are eigenvalues of $S_1$ and $\tilde{\beta}_i$ is the main angle of $S_1$ defined in Section~2.
See Theorem~\ref{thm:2} and Remark~\ref{rem:1} for the spectra of doubly regular tournaments.

In Section 2, we prepare the fundamental notation of oriented graphs and characterize the tournament with the adjacency matrix having a certain spectrum in terms of
the spectrum of the Seidel matrix.
In Section 3, we prove Theorem~\ref{thm:1}.

\section{Tournaments and their Seidel matrices}
Let $G=(V,E)$ be an oriented graph of size $n$, namely the vertex set $V$ consists of $n$ elements and the edge set $E\subset V\times V$ satisfies $E\cap E^T=\emptyset$, where $E^T:=\{(x,y)\mid (y,x)\in E\}$. 
The {\it adjacency matrix} $A$ of $G$ is indexed by the vertex set $V$, and its entries are defined as follows: 
\[
A_{x y}=
\begin{cases}
1 \text{ if $(x,y) \in E$}, \\
0 \text{ otherwise}.
\end{cases}
\]
Since $E$ satisfies $E\cap E^T=\emptyset$, $A$ satisfies $A\circ A^T=\boldsymbol{0}$, where $\circ$ is the entrywise product of matrices and $\boldsymbol{0}$ denotes the zero matrix. 
The {\it Seidel matrix} $S$ of $G$ is defined by $S=\sqrt{-1}(A-A^T)$.
An oriented graph $G$ is said to be a {\it tournament} if its adjacency matrix satisfies $A+A^T=J-I$.
The vector $A \boldsymbol{1}$ is called the score vector of the tournament,
where $\boldsymbol{1}$ is the all-ones column vector.
A tournament $G$ of size $n$ is {\it regular} if all entires of the score vector are equal to $(n-1)/2$, which implies that $n$ must be odd.
A regular tournament $G$ is {\it doubly regular} if the number of neighbors of a pair of distinct vertices does not depend on the choice of the pair. 
A tournament $G$ of size even $n$ is {\it almost regular} if the entires of the column vector $A \boldsymbol{1}$ consist of $n/2$ and $(n-2)/2$ with the same number of times $n/2$.

A square matrix $M$ is said to be {\it normal} if $MM^T=M^T M$. 
It is known that a normal matrix $M$ can be diagonalized by a unitary matrix, 
namely the eigenspaces corresponding to different eigenvalues are orthogonal.
Let $\tau_i$ ($1\leq i\leq s$) be the distinct eigenvalues of $M$. 
Let $m_i$ be the multiplicity of $\tau_i$, and $\mathcal{E}_i$ the eigenspace of $\tau_i$. 
Let $P_i$ be the orthogonal projection matrix onto $\mathcal{E}_i$. 
Then we can easily find that $P_i^*=P_i$, $\sum_{i=1}^s P_i=I$ and $P_i P_j=\delta_{i j}P_i$, where $P_i^*$ denotes the transpose conjugate of $P_i$  and $\delta_{i j}$ denotes the Kronecker delta.
The notion of main angles for the adjacency matrix of a simple undirected graph was introduced in \cite{C1971}, see \cite{R2007} for the recent progress.
Here we consider the same concept for a normal matrix.
Define $\beta_i$ by 
\[
\beta_i:= \frac{1}{\sqrt{n}} \sqrt{(P_i \cdot \boldsymbol{1})^*(P_i \cdot \boldsymbol{1})},
\] 
where $n$ is the size of the square matrix $M$.
We call $\beta_i$ the {\it main angle} of $\tau_i$.
By the definition of main angles, we have 
\begin{align}\label{eq:01}
\sum_{i=1}^s \beta_i^2=1.
\end{align}

Let $G$ be a tournament of size $n$ with the adjacency matrix $A$ and the Seidel matrix $S$.
Let $\{\theta_i\}_{i=1}^s$ be the distinct eigenvalues of $A$.
Let $m_i$ be the algebraic multiplicity of $\theta_i$.
When $A$ is normal, we denote the main angle of $\theta_i$ by $\beta_i$.

Since the Seidel matrix $S$ is normal, we may define the main angles of $S$.
Moreover $S$ is Hermitian, and all eigenvalues of $S$ are real.
Let $\tilde{\theta}_1>\cdots>\tilde{\theta}_{\tilde{s}}$ be the distinct eigenvalues of $S$ and 
let $\tilde{m}_i$, $\tilde{\beta}_i$ be the multiplicity and the main angle of $\tilde{\theta}_i$ for $1\leq i\leq \tilde{s}$.
The spectral decomposition of the Seidel matrix is $S=\sum_{i=1}^{\tilde{s}}\tilde{\theta}_i P_i$.
The following are fundamental results on $S$.
\begin{lemma}\label{lem:1}
Let $G$ be a tournament of size $n$ with the adjacency matrix $A$ and the Seidel matrix $S$.
\begin{enumerate}
\item $\tilde{\theta}_{\tilde{s}+1-i}=-\tilde{\theta}_{i}$, $\tilde{m}_{\tilde{s}+1-i}=\tilde{m}_{i}$ and $\tilde{\beta}_{\tilde{s}+1-i}=\tilde{\beta_{i}}$ for $1\leq i \leq \tilde{s}$,
\item $\sum_{i=1}^{\tilde{s}}\tilde{m}_i=n$ and $\sum_{i=1}^{\tilde{s}}\tilde{m}_i \tilde{\theta}_i^2=n^2-n$.
\item $G$ is regular if and only if $S\boldsymbol{1}=0$.
\end{enumerate}
\end{lemma}
\begin{proof}
(1) Follows from that $\sqrt{-1}S$ is skew-symmetric.

(2) Follows from taking the traces of $\sum_{i=1}^{\tilde{s}}P_i=I$ and $\sum_{i=1}^{\tilde{s}}\tilde{\theta_i}^2P_i=S^2=-A^2-(A^T)^2+A A^T+A^T A$.

(3) $G$ is regular if and only if $A\boldsymbol{1}=\frac{n-1}{2}\boldsymbol{1}$.
When this is the case, by $A^T=J-A-I$ we have $A^T\boldsymbol{1}=\frac{n-1}{2}\boldsymbol{1}$.
Thus $A\boldsymbol{1}=\frac{n-1}{2}\boldsymbol{1}$ if and only if $S\boldsymbol{1}=0$.
\end{proof}
 
The following lemma characterizes the almost regularity for a tournament in terms of the data of the spectrum of the Seidel matrix.
\begin{lemma}\label{lem:almost}
Let $n$ be an even integer at least four and $G$ a tournament of size $n$.
The the following are equivalent:
\begin{enumerate} 
\item $G$ is almost regular,
\item $\sum_{i=1}^{\tilde{s}}\tilde{\theta_i}^2 \tilde{\beta_i}^2=1$.
\end{enumerate}
\end{lemma} 
\begin{proof}
Denote the score vector of $A$ by $\boldsymbol{s}$ and the $i$-th entry of $\boldsymbol{s}$ by $s_i$.
First the following equality holds for any tournament:
\begin{align}\label{eq:2}
\sum_{i=1}^n s_i=\boldsymbol{s}^T \boldsymbol{1}=\boldsymbol{1}^T A \boldsymbol{1}=\tfrac{1}{2}\boldsymbol{1}^T(A+A^T)\boldsymbol{1}=\tfrac{1}{2}\boldsymbol{1}^T(J-I)\boldsymbol{1}=\tfrac{n(n-1)}{2}.
\end{align}
Since the size of the tournament is even, $\frac{n}{4}\leq\sum_{i=1}^n (s_i-\frac{n-1}{2})^2$ holds, namely 
\begin{align} \label{eqn:score}
\boldsymbol{s}^T \boldsymbol{s}\geq \tfrac{n(n^2-2n+2)}{4}.
\end{align}
Equality holds in \eqref{eqn:score} if and only if the score vector consists of $n/2$ and $(n-2)/2$ with the same number of times $n/2$, that is, $G$ is almost regular.

We calculate $\boldsymbol{s}^T \boldsymbol{s}$ in terms of the data of the Seidel matrix.
From $S=\sqrt{-1}(2A-J+I)$ we have $\boldsymbol{s}=\tfrac{1}{2}((n-1)\boldsymbol{1}-\sqrt{-1}S\boldsymbol{1})$.
Since 
\begin{align*}
\boldsymbol{s}^T \boldsymbol{s}&=\frac{1}{4}((n-1)^2 \boldsymbol{1}^T \boldsymbol{1}+\boldsymbol{1}^T S^2 \boldsymbol{1})\\
&=\frac{1}{4}(n(n-1)^2+n\sum_{i=1}^{\tilde{s}}\tilde{\theta_i}^2 \tilde{\beta_i}^2),
\end{align*} 
$G$ is almost regular if and only if $\sum_{i=1}^{\tilde{s}}\tilde{\theta_i}^2 \tilde{\beta_i}^2=1$.
\end{proof}

For a square matrix $A$, we denote the characteristic polynomial of $A$ by $P_A(x)$, that is $P_A(x)=\det(A-x I)$. 
We use the following lemma to prove Theorem \ref{thm:1}.
\begin{lemma}\label{lem:char}
Let $M$ be a normal matrix, $\tau_i$ the distinct 
eigenvalues of $M$, and $\beta_i$ the main angle of $\tau_i$.
Let $c$ be a complex number. 
Then
\[
P_{M+c J}(x)=P_M(x)\big( 1+c \sum_{i=1}^s  \frac{n \beta_i^2}{\tau_i-x} \big). 
\]
\end{lemma}
\begin{proof}
A normal matrix has the spectral decomposition $M=\sum_{i=1}^s \theta_i P_i$, where $P_i$ is the orthogonal projection onto the eigenspace for the eigenvalues $\tau_i$.
By $\sum_{i=1}^s P_i=I$ we have $(M-x I)^{-1}=\sum_{i=1}^s \frac{P_i}{\tau_i-x}$.
Thus 
\begin{align*}
P_{M+c J}(x)&=\det (M+c J-x I)\\
&=\det (M-x I)+c \boldsymbol{1}^T \adj (M-x I) \boldsymbol{1}\\
&=P_M(x)(1+c \boldsymbol{1}^T (M-x I)^{-1} \boldsymbol{1}) \\
&=P_M(x)(1+c \sum_{i=1}^s \frac{\boldsymbol{1}^T P_i \boldsymbol{1}}{\tau_i-x} ) \\
&=P_M(x)(1+c \sum_{i=1}^s \frac{n \beta_i^2}{\tau_i-x}), 
\end{align*}
where $\adj(*)$ is the adjugate matrix.  
\end{proof}
Applying Lemma~\ref{lem:char} to the Seidel matrix of a tournament, we have the following corollary: 
\begin{corollary}\label{cor:1}
Let $G$ be a tournament of size $n$ with the adjacency matrix $A$ and the Seidel matrix $S$.
Then the following holds:
\begin{align}
P_A(x)=(\tfrac{-\sqrt{-1}}{2})^n P_S (\sqrt{-1}(2x+1))(1+\sqrt{-1}\sum_{i=1}^{\tilde{s}} \tfrac{n \tilde{\beta}_i^2}{\tilde{\theta}_i-\sqrt{-1}(2x+1)}).
\end{align} 
\end{corollary}
\begin{proof}
Since $A=\frac{1}{2}(-\sqrt{-1}S-I+J)$ holds, applying Lemma \ref{lem:char} yields the following equations;
\begin{align*}
P_A(x)&=\det (A-x I)\\
&=(\tfrac{-\sqrt{-1}}{2})^n \det (S+\sqrt{-1}J-\sqrt{-1}(2x+1)I)\\
&=(\tfrac{-\sqrt{-1}}{2})^n P_S (\sqrt{-1}(2x+1))(1+\sqrt{-1}\sum_{i=1}^{\tilde{s}} \tfrac{n \tilde{\beta}_i^2}{\tilde{\theta}_i-\sqrt{-1}(2x+1)}). \qquad \qedhere
\end{align*}
\end{proof}

\begin{theorem}\label{thm:2}
Let $G$ be a tournament of size $n$ with the adjacency matrix $A$ and the Seidel matrix $S$.
Then the following are equivalent:
\begin{enumerate}
\item $G$ is doubly regular, 
\item $A$ satisfies that $s=3$ and 
$
(\theta_i)_{i=1}^3=(\tfrac{n-1}{2},\tfrac{-1+\sqrt{-n}}{2}, \tfrac{-1-\sqrt{-n}}{2}) $, 
\item $S$ satisfies that $\tilde{s}=3$, 
$
(\tilde{\theta_i})_{i=1}^3=(\sqrt{n},0,-\sqrt{n})$, and $(\tilde{\beta_i})_{i=1}^3=(0,1,0)$.
\end{enumerate}
\end{theorem}
\begin{proof}
$(1)\Leftrightarrow (2)$:
The equivalence is proven in \cite[Theorem~3.2]{Z}.

$(1),(2)\Rightarrow (3)$:
Note that by \cite{Z} $P_A(x)=-(x-\tfrac{n-1}{2})(x^2+x+\tfrac{n+1}{4})^{\tfrac{n-1}{2}}$ and $A$ is a normal matrix. 
Since $G$ is regular, the main angles of $A$ are given by $(\beta_i)_{i=1}^3=(1,0,0)$.
Applying Lemma \ref{lem:char} yields the following equation:
\begin{align*}
P_S(x)=-x(x^2-n)^{\tfrac{n-1}{2}}.
\end{align*}
Since $G$ is regular, $\tilde{\beta_2}$ and $\tilde{\beta_3}$ are zero, and thus $\tilde{\beta_1}$ is one.

$(3)\Rightarrow (2)$:
By Lemma~\ref{lem:1} (1) and (2), $(\tilde{m_i})_{i=1}^3=(\tfrac{n-1}{2},1,\tfrac{n-1}{2})$.
Then it follows from Corollary~\ref{cor:1}. 
\end{proof}
\begin{remark}\label{rem:1}
When $G$ is a doubly regular tournament, $A$ is a normal matrix.
Then the multiplicities of eigenvalues for the adjacency matrix and the Seidel matrix are given by
$(m_i)_{i=1}^3=(1,\tfrac{n-1}{2},\tfrac{n-1}{2})$, $(\beta_i)_{i=1}^3=(1,0,0)$
and $(\tilde{m}_i)_{i=1}^3=(\tfrac{n-1}{2},1,\tfrac{n-1}{2})$.
\end{remark}

\begin{theorem}\label{thm:3}
Let $G$ be a tournament of size $n-1$ with the adjacency matrix $A$ and the Seidel matrix $S$.
Then the following are equivalent:
\begin{enumerate}
\item $s=4$, $(\theta_i)_{i=1}^4=(\tfrac{-1+\sqrt{-n}}{2},\tfrac{-1-\sqrt{-n}}{2},\tfrac{n-3+\sqrt{(n-3)(n+1)}}{4}, \tfrac{n-3-\sqrt{(n-3)(n+1)}}{4}),$
\item $\tilde{s}=4$, $(\tilde{\theta_i})_{i=1}^4=(\sqrt{n},1,-1,-\sqrt{n}),(\tilde{\beta_i})_{i=1}^4=(0,\tfrac{1}{\sqrt{2}},\tfrac{1}{\sqrt{2}},0)$.
\end{enumerate}
\end{theorem}
\begin{proof}
$(1)\Rightarrow (2)$:
The algebraic multiplicities of $\theta_1$ and $\theta_2$ ($\theta_3$ and $\theta_4$) are equal since they are algebraically conjugate.
Define $m_1$ (resp.\ $m_3$) by the algebraic multiplicity of $\theta_1$ (resp.\ $\theta_3$).
Since the size of the matrix $A$ is $n-1$ and the trace of $A$ is $0$, we have 
\begin{align*}
2m_1+2m_3=n-1,\\ -m_1+\tfrac{n-3}{2}m_3=0.
\end{align*}  
These equations yield $m_1=\tfrac{n-3}{2}$, $m_3=1$.
By \cite[Lemma 1(i)]{KS1994}, all eigenvectors of $A$ for eigenvalue $\theta_i$ for $i=1,2$ are also eigenvectors of $S$ with eigenvalue $-2\text{Im}\theta_i$.
Moreover by \cite[Lemma 1(i)]{KS1994} the corresponding main angles are zero. 
Since the dimension of the subspace of $\mathbb{C}^{n-1}$ spanned by those eigenvectors is $2m_1=n-3$ and $S$ is skew-symmetric, 
we set the remaining eigenvalues of $S$ as $\tau,-\tau$, where $\tau$ is a non-negative real number.
By Lemma~\ref{lem:1} (2), we obtain $n(n-3)+2\tau^2=(n-1)(n-2)$.
Thus $\tau=1$.
We renumber $(\tilde{\theta_i})_{i=1}^4=(\sqrt{n},1,-1,-\sqrt{n})$.
By \eqref{eq:01} and Lemma~\ref{lem:1} (1),
$(\tilde{\beta}_i)_{i=1}^4=(0,\tfrac{1}{\sqrt{2}},\tfrac{1}{\sqrt{2}},0)$.

$(2)\Rightarrow (1)$:
By Lemma~\ref{lem:1} (1) and (2), $(\tilde{m_i})_{i=1}^4=(\tfrac{n-3}{2},1,1,\tfrac{n-3}{2})$.
Then it follows from Corollary~\ref{cor:1}.
\end{proof}
\begin{remark}\label{rem:2}
\begin{enumerate}
\item When $G$ is a tournament satisfying the conditions in Theorem~\ref{thm:3}, then the algebraic multiplicities of the eigenvalues for the adjacency matrix and the Seidel matrix are given by
\begin{align}
(m_i)_{i=1}^4=(\tfrac{n-3}{2},\tfrac{n-3}{2},1,1), (\tilde{m_i})_{i=1}^4=(\tfrac{n-3}{2},1,1,\tfrac{n-3}{2})
\end{align}
\item As will be shown in the next section,
the tournament of size $n-1$ considered in Theorem~\ref{thm:3} is obtained from a doubly regular tournament of size $n$.
\end{enumerate}
\end{remark}

\section{Proof of Theorem~\ref{thm:1}}
\begin{proof}[Proof of Theorem~\ref{thm:1}]
Let $G$ be a doubly regular tournament of size $n$ with the adjacency matrix $A$.
Take a vertex $x$ in $G$.
Let $A_1$ be the adjacency matrix of the graph obtained by deleting the vertex $x$ from $G$,
namely after reordering of the vertices of the tournament $G$ we have 
\begin{align*}
A=\begin{pmatrix}
0  & \boldsymbol{v}^T \\
\boldsymbol{1}-\boldsymbol{v} & A_1   
\end{pmatrix},
\end{align*}
where $\boldsymbol{v}$ is a $(0,1)$-column vector.
We calculate the spectrum of $A_1$.
Let $\tau_1\geq\cdots\geq \tau_{n-1}$ be all the eigenvalues of the Seidel matrix $S_1=\sqrt{-1}(A_1-A_1^T)$.
The interlacing eigenvalues theorem for bordered matrices \cite[Theorem 4.3.8]{HJ} shows that 
\begin{align*}
\tau_1&=\cdots=\tau_{2k}=\sqrt{n},\\
\tau_{2k+3}&=\cdots=\tau_{n-1}=-\sqrt{n}.
\end{align*} 
Next we determine $\tau_{2k+1}$ and $\tau_{2k+2}$.
By Lemma~\ref{lem:1} (1) $\tau_{2k+1}=-\tau_{2k+2}$.
And by Lemma~\ref{lem:1} (2) $\tau_{2k+1}^2+\tau_{2k+2}^2+2\dot (\tfrac{n-1}{2}-1)n=n^2-3n+2$. 
Thus $\tau_{2k+1}=-\tau_{2k+2}=1$ as desired.
Then we can easily find that $\boldsymbol{y}=\boldsymbol{1}+(\sqrt{-1}-1)\boldsymbol{v}$ is the eigenvector of $S_1$ with eigenvalue $1$, and its conjugate vector is that of $S_1$ with the eigenvalue $-1$.
Now we denote by $\tilde{\theta}_1>\cdots>\tilde{\theta}_4$ the distinct eigenvalues of $S_1$, $\beta_i$ ($i=1,\ldots,4$) the corresponding main angles.
Then direct calculation of the norm of $\boldsymbol{y}^T\boldsymbol{1}$ and $\bar{\boldsymbol{y}}^T\boldsymbol{1}$, where $\bar{\boldsymbol{y}}$ denotes the complex conjugate vector of $\boldsymbol{y}$, shows that $\beta_2=\beta_3=\tfrac{1}{\sqrt{2}}$.
By \eqref{eq:01} $\beta_1=\beta_4=0$ hold.

Conversely let $G_1$ be a tournament of size $n-1$ with the adjacency matrix $A_1$ and the Seidel matrix $S_1$ satisfying the property \eqref{eq:1}. 
By Remark~\ref{rem:2} (1) the multiplicities of $S_1$ are $(\tilde{m_i})_{i=1}^4=(\tfrac{n-3}{2},1,1,\tfrac{n-3}{2})$.
It follows from Lemma~\ref{lem:almost} that $G_1$ is almost regular.

Hence we can add one more vertex to $G_1$ so that it becomes a regular tournament $G$ of size $n$.
Let $S$ be the Seidel matrix of $G$.
We may express 
\begin{align*}
S=\begin{pmatrix}
0  & \boldsymbol{w}^T \\
-\boldsymbol{w} & S_1   
\end{pmatrix}
\end{align*}
for some $(\sqrt{-1},-\sqrt{-1})$-column vector $\boldsymbol{w}$ such that $\boldsymbol{w}^T\boldsymbol{1}=0$ and $S_1 \boldsymbol{1}=\boldsymbol{w}$.
Then by Lemma~\ref{lem:char} we have
\begin{align*}
P_{S}(t)&=\det \begin{pmatrix}
-t  & \boldsymbol{w}^T \\
-\boldsymbol{w} & -t I+S_1   
\end{pmatrix}\displaybreak[0] \\
&=\det \begin{pmatrix}
-t  & \boldsymbol{w}^T \\
-t \boldsymbol{1} & -t I+S_1   
\end{pmatrix}\displaybreak[0] \\
&=\det \begin{pmatrix}
-n t  & -t \boldsymbol{1}^T \\
-t \boldsymbol{1} & -t I+S_1   
\end{pmatrix}\displaybreak[0] \\
&=t\det \begin{pmatrix}
-n  & -\boldsymbol{1}^T \\
0 & -t I+S_1+\tfrac{t}{n}J   
\end{pmatrix}\displaybreak[0] \\
&=(-n)t P_{S_1+\tfrac{t}{n}J}(t) \displaybreak[0] \\
&=(-n)t P_{S_1}(t)(1+\tfrac{(n-1) t}{n}\sum_{i=1}^4\tfrac{\tilde{\beta}_i^2}{\tilde{\theta}_i-t}) \displaybreak[0] \\
&=(-n)t(t^2-n)^{\tfrac{n-1}{2}}(t^2-1)(1+\tfrac{(n-1) t}{n}(\tfrac{1/2}{-1-t}+\tfrac{1/2}{1-t})) \displaybreak[0] \\
&=-t(t^2-n)^{\tfrac{n-3}{2}}.
\end{align*}
Since $G$ is regular, the main angle corresponding to the eigenvalue $0$ is one and the others are zero. 
Therefore $G$ is a doubly regular tournament by Theorem~\ref{thm:2}.
\end{proof}

\noindent
{\it Hiroshi Nozaki}\\
	Graduate School of Information Sciences, \\
	Tohoku University \\
	Aramaki-Aza-Aoba 6-3-09, \\
	Aoba-ku, Sendai 980-8579, \\
	Japan.\\ 
	nozaki@ims.is.tohoku.ac.jp\\
	
	\qquad \\
{\it Sho Suda}\\
	Graduate School of Information Sciences, \\
	Tohoku University \\
	Aramaki-Aza-Aoba 6-3-09, \\
	Aoba-ku, Sendai 980-8579, \\
	Japan.\\ 
	suda@ims.is.tohoku.ac.jp\\


\begin{thebibliography}{99}
\bibitem{C1971}
D. Cvetkovi\'{c}, Graphs and their spectra. Univ.\ Beograd.\ Publ.\ Elektrotehn.\ Fak.\ Ser.\ Mat.\ Fiz.\ No.\ 354--356 (1971), 1--50.

\bibitem{CRS1997}
D. Cvetkovi\'{c}, P. Rowlinson and S. Simi\'{c}, 
Eigenspaces of graphs. Encyclopedia of Mathematics and its Applications, 66. Cambridge University Press, Cambridge, 1997. xiv+258 pp.

\bibitem{GR2001}
C. Godsil and G. Royle, Algebraic graph theory. Graduate Texts in Mathematics, 207. Springer-Verlag, New York, 2001. xx+439 pp.

\bibitem{KS1994}
Kirkland, Stephen J.and Shader, Bryan L.
Tournament matrices with extremal spectral properties. (English summary) 
Linear Algebra Appl.\ 196 (1994), 1--17. 

\bibitem{HJ}
R. A. Horn, C. R. Johnson, Matrix Analysis,
{Cambridge University Press,\ } Cambridge, 1990.

\bibitem{KB1994}
S. J. Kirkland and B. L. Shader,  Tournament matrices with extremal spectral properties. Linear Algebra Appl. 196 (1994), 1-17.

\bibitem{KS2008}
C. Koukouvinos and S. Stylianou,  On skew-Hadamard matrices. Discrete Math. 308 (2008), no. 13, 2723--2731.

\bibitem{RB1972}
K. B. Reid and E. Brown, Doubly regular tournaments are equivalent to skew Hadamard matrices. J. Combinatorial Theory Ser. A 12 (1972), 332--338.

\bibitem{R2007}
P. Rowlinson, 
The main eigenvalues of a graph: a survey,
Appl. Anal. Discrete Math. 1 (2007), no. 2, 445-471.

\bibitem{Z}
N. Zagaglia Salvi,
Some properties of regular tournament matrices. (Italian. English summary) Proceedings of the conference on combinatorial and incidence geometry: principles and applications (La Mendola, 1982), 635--643, 
Rend. Sem. Mat. Brescia, 7, Vita e Pensiero, Milan, 1984.

\end{thebibliography}
\end{document}